\documentclass[AIF,Unicode,manuscript]{cedram}
\usepackage{amssymb,epsfig,newlfont,amsmath}
\usepackage{color}
\usepackage{amssymb,epsfig,newlfont,amsmath}
\usepackage[french,english]{babel}
\usepackage[T1]{fontenc}
\usepackage{amssymb}
\usepackage{amsmath}
\usepackage{amscd}

\catcode`\Ž=13
\catcode`\=13
\catcode`\ˆ=13
\catcode`\=13
\catcode`\=13
\catcode`\‰=13
\catcode`\™=13
\catcode`\"=13
\catcode`\•=13
\catcode`\ž=13
\catcode`\=13
\catcode`\'=13
\catcode`\Ï=13
\catcode`\Ë=13

\def Ž{\'e}
\def {\`e}
\def ˆ{\`a}
\def {\`u}
\def {\^e}
\def ‰{\^a}
\def ™{\^o}
\def "{\^{\i}}
\def •{\''{\i}}
\def ž{\^u}
\def {\c c}
\def '{\''e}
\def Ï{\oe}
\def Ë{\`A}

\def\math#1{{\mathbb{#1}}}

\def\CM{\math C}

\def\bs{\boldsymbol}
\def\bsl{\backslash}
\def\mun{^{-1}}

\def\goth{\mathfrak} 
\def\ga{\goth a}
\def\trunc{\mathbf{\Lambda}^T}
\def\ptf{\,\,.}
\def\com#1{\!\!\quad\hbox{#1}\quad\!\!}
\def\coma#1{\,\,\hbox{#1}\,\,}
\def\comm#1{\qquad\hbox{#1}\qquad}

\def\fp{{\mathfrak p}}
\def\fq{{\mathfrak q}}
\def\hmu{\mu}
\def\omu{\lambda}
\def\hD{\widehat\Delta}
\def\oD{\Delta}
\def\hF{\hat F}
\def\oF{F}
\def\hE{ \hat E}
\def\E{ E}

\def\hg{{\widehat\eta}}

\def\ob{b}
\def\hb{{\widehat b}}
\def\og{\eta}
\def\mT{\theta}
\def\oT{{\mT}}
\def\hT{\widehat{\oT}}
\def\oV{V}
\def\hV{U}

\def\lb{\langle}
\def\rb{\rangle}

\def\D{D}
\def\L{\Lambda}
\def\Lu{{\Lambda_1}}
\def\Ld{{\Lambda_2}}
\def\Ld{{\L_2}}

 \title{A pleasant exercise : Langlands' Boulder Lemma}
\author{\firstname{Jean-Pierre} \lastname{Labesse}}
\address{Institut Math\'ematique de Luminy\\ UMR 7373}

\begin{document}

\begin{abstract}In the Proceedings of the AMS Boulder conference in 1965 Langlands states a combinatorial lemma 
involving families of characteristic functions attached to ordered partitions of an obtuse basis in a
finite dimensional euclidean vector space. Langlands does not give any indication about the proof of
the lemma which is said to be a ``pleasant exercise''. 
Since we did not find a proof in the literature we decided to give one. We believe it of some
interest for the history of the subject.
\end{abstract}

\begin{altabstract}Dans les Proceedings de la confŽrence de l'AMS ˆ Boulder en 1965 Langlands 
Žnonce un lemme combinatoire qui utilise des familles de fonctions caractŽristiques attachŽes
ˆ des partitions ordonnŽes d'une base obtuse dans un espace vectoriel euclidien de dimension finie.
 Langlands ne donne aucune indication sur la preuve de ce lemme qu'il dŽclare tre un ``plaisant exercice''.
Comme il ne semble pas exister de preuve dans la littŽrarture nous avons dŽcidŽ d'en donner une.
Nous pensons qu'elle peut tre intŽressante au moins du point de vue de l'histoire du domaine.
\end{altabstract}

\thanks{I am indebted to Bill Casselman for suggesting that we look at these questions. We discussed 
preliminary results toward a proof of the Boulder combinatorial lemma 
and also some related historical questions, in particular the early proofs for the scalar product
formula of two truncated Eisenstein series
by Selberg and others, in relative rank one, using Maa\ss-Selberg relations. This
would be the subject of another historical study and will not be discussed here.
I also thank J.-L. Waldspurger for comments on a preliminary version and suggesting a simpler proof
of Corollary~\ref{combi} which eventually was recognized as a particular case of \ref{combo}.}

\maketitle

\section{Introduction}

\subsection*{A combinatorial lemma}
Some months ago Bill Casselman observed that there does not seem to exist any proof 
in the literature for a combinatorial lemma stated by Langlands in \cite[Section 8]{L2}
 we shall call the Boulder lemma. This
is the first occurence of a family of combinatorial statements that became, 
after Langlands article \cite{L3} and Arthur's early contributions (\cite{A1}, \cite{A2}, \cite{A3}),
a ubiquitous tool for harmonic analysis on reductive groups over local fields and,
over adles, for the Trace Formula. In  \cite[Section 9]{L2}
the Boulder lemma is used by Langlands to define and control avatars, denoted $E''$, of Eisenstein
series arising from cuspidal forms and he gives a formula for the scalar product 
of two such $E''$-series. This is a generalization of a formula due to Selberg in rank one. It
allows Langlands to outline a more direct proof than in \cite{L1} of the
meromorphic continuation of those Eisenstein series in arbitrary rank.
Langlands also anticipated that the scalar product formula would be an essential tool to establish 
the Trace Formula for reductive groups of arbitrary rank. 
In \cite[Lemma 4.2]{A3} Arthur proves a formula for the scalar product of two truncated Eisenstein series
(arising from cuspidal forms as above) and attributes the formula to Langlands. This implicitly assumes that
series $E''$ are nothing but truncated Eisenstein series.
This is not immediate since ordered partitions used by Langlands never occur in Arthur's statement.
At least from an historical point of view, it seems useful to give a proof for the Boulder lemma,
to relate it to the modern treatment of the combinatorics and, in particular, to establish
the equivalence alluded to above.

\subsection*{The contents}
We state the Boulder lemma with Langlands' notation as Proposition~\ref{plex}.
Our proof, given in section \ref{prob}, relies on matrix equations established 
in section \ref{matrix} that are variants of combinatorial identities mainly due 
to Arthur inspired by Langlands work. Our Proposition \ref{combo} generalizes the well known 
``Langlands' combinatorial lemma''  (see \cite[Lemma 2.3]{A1}  and \cite[Lemma 6.3]{A2}).
This combinatorial lemma appears here as Corollary~\ref{part};
we observe that a particular case of it (when the parameter $\L$ is 
in the positive chamber) is a key ingredient for the Langlands classification 
of admissible irreducible representations of reductive groups over local fields. It
appears, I believe for the first time, in \cite[page 156]{L3}.
In fact this paper was circulated, as a preprint, a long time before its actual publication
(see also \S6 of Chapter IV in \cite{BW}). Now, Corollary \ref{combi} is essentially equivalent to
Proposition~\ref{simple} which together with \ref{part} yields a proof of the Boulder lemma.
As another corollary of Proposition~\ref{simple}
we show in Section \ref{trunc} that Langlands' series $E''$ do coincide with Arthur's truncated 
Eisenstein series and hence that Langlands' formula for the scalar product of $E''$-series is
equivalent to the formula for the scalar product of truncated Eisenstein series
proved by Arthur. We believe that Arthur knew about Proposition~\ref{simple} since this 
is implicit in his attribution to Langlands of the scalar product formula.

\section{The Boulder lemma}\label{orig}

\subsection*{Ordered partitions}
Consider a euclidean vector space $V$ of dimension $p\ge1$. The euclidean structure
allows to identify $V$ with its dual $V'$. Let $\oD$ be 
a basis
and $\hD$ the dual basis. We denote by 
$$\D:\omu\in\oD\mapsto\D(\omu)=\hmu\in\hD$$ 
the canonical bijection between $\oD$ and $\hD$.
Let $\fp$ be an ordered partition of $\oD$ built out of $r=r(\fp)$ non empty 
disjoint subsets $\oF_\fp^u$ with cardinals $a^u_\fp$ with $1\le u\le r$ so that
$$\oD=\bigcup_{u=1}^{u=r} \oF_\fp^u\comm{and} \sum_\fp a^u_\fp=p\ptf$$
This is equivalent to be given an increasing filtration of $\oD$ by subsets $\E_\fp^v$
$$\emptyset= \E_\fp^0\subsetneq \E_\fp^1\subsetneq \E_\fp^2\subsetneq\cdots\subsetneq \E_\fp^r=\oD
\comm{ where} \E_\fp^v=\bigcup_{u\le v} \oF_\fp^u\ptf$$
This defines an increasing filtration
of $V$ by subspaces $\hV_\fp^{v}$ generated by $ \hE_\fp^v=\D(\E_\fp^v)$; 
$$\{0\}=\hV_\fp^0\subsetneq \hV_\fp^1\subsetneq \hV_\fp^2\subsetneq \cdots\subsetneq \hV_\fp^r=V$$
and a decreasing filtration by subspaces
 $\oV_\fp^{v}$  generated by $\oD-\E_\fp^v$:
 $$V=\oV_\fp^0\supsetneq \oV_\fp^1\supsetneq \oV_\fp^2\supsetneq \cdots\supsetneq \oV_\fp^r=\{0\}.$$
For any $v\in\{0,1\cdots, r\}$ 
we have an orthogonal decomposition:
$V=\hV_\fp^v\oplus V_\fp^v\ptf$
Let $W_\fp^u$ be the orthogonal supplement to $\hV_\fp^{u-1}$ in $\hV_\fp^u$
or, equivalently to $\oV_\fp^{u}$ in $\oV_\fp^{u-1}$:
$$\hV_\fp^{u-1}\oplus W_\fp^u=\hV_\fp^u
\comm{}\oV_\fp^{u}\oplus W_\fp^u=\oV_\fp^{u-1}\ptf
$$
One has an orthogonal direct sum decomposition of $V$:
$$V=\bigoplus_{u=1}^{u=r}W_\fp^u\ptf$$
For $\omu\in \oF_\fp^u$ denote by $\omu_\fp$ the orthogonal projection of $\omu$ on $W_\fp^u$; 
the $\omu_\fp$ build a basis $\oD_\fp^u$ of it. Similarly
for $\hmu\in \hF_\fp^u:=\D(\oF_\fp^u)$ denote by $\hmu_\fp$ the orthogonal projection of $\hmu$ on $W_\fp^u$,
they build a basis $\hD_\fp^u$.
Since $\hV_\fp^u$ is orthogonal to $V_\fp^u$ then, for any $\omu\in \oF_\fp^u$ and any $\hmu\in \hF_\fp^u$
one has
$$\lb\hmu,\omu\rb=\lb\hmu,\omu_\fp\rb=\lb\hmu_\fp,\omu\rb=\lb\hmu_\fp,\omu_\fp\rb\ptf$$ 
This shows that the sets $\oD_\fp^u$ and $\hD_\fp^u$ are dual basis. 
Denote by $\oD_\fp$ the union of the $\oD_\fp^u$ (resp. $\hD_\fp$ the dual basis) and 
$$\D_\fp: \oD_\fp\to\hD_\fp$$
the natural bijection. 

\subsection*{Functions $\phi_\fp^\L$ and $\psi_\fp^\L$}
Let $H\mapsto\phi_\fp^\L(H)$ be the characteristic function of the cone 
in $V$~\footnote{Although we have identified $V$ with its dual we shall look at elements in 
$\oD_\fp$ and $\hD_\fp$ as linear forms
rather than vectors. Langlands does not indentify $V$ with $V'$; he
considers $H$ as an element of $V$ and $\L$ 
as an element in $V'$. Here we emphasis a kind of symmetry between $H$ and $\L$.} 
defined by
$$\omu_\fp(H)\le0\com{if}\hmu_\fp(\L)>0\com{and}\omu_\fp(H)>0\com{otherwise}
$$
for all $\omu\in\oD$ (or equivalently for all $\omu_\fp\in\oD_\fp$)
and where $\hmu_\fp=\D_\fp(\omu_\fp)$.
Now let $H\mapsto\psi_\fp^\L(H)$ be the characteristic function of the cone 
$\omu_\fp(H)>0\quad\hbox{for $\omu\in F_\fp^1$}$
and
$$\omu_\fp(H)\le0\com{if}\hmu_\fp(\L)>0\com{and}\omu_\fp(H)>0\com{if}\hmu_\fp(\L)\le0
\com{for}\omu\notin F_\fp^1$$ where $\hmu_\fp=\D(\omu_\fp)$.
Let $b^\L_\fp$ be the number of $\hmu\in\hD$ such that $\hmu_\fp(\L)\le0$
and $c^\L_\fp$ the number of those $\hmu$ such that moreover $\hmu\notin\hF_\fp^1$.
Finally Langlands defines integers
$$\alpha^\L_\fp=b^\L_\fp+\sum_{u=1}^{u=r} (a^u_\fp+1)\quad \hbox{and}\quad
\beta^\L_\fp=1+c^\L_\fp+\sum_{u=2}^{u=r} (a^u_\fp+1)\ptf$$
An element $\L\in V$ is said to be ``in the positive chamber'' (with respect to $\oD$) if $\omu(\L)>0$ for 
all $\omu\in\oD$. Let $d_\oD(\L)=1$ if $\L$ is in the positive chamber
and zero otherwise.  We may now state the Boulder lemma 
\cite[Section 8]{L2}~\footnote{We observe that Langlands defines  $\phi_\fp^\L(H)$ and $\psi_\fp^\L(H)$
only for $H$ and $\L$ regular i.e. for $H$ outside the walls defined by the $\omu_\fp$
and $\L$ outside the walls defined by the $\hmu_\fp$ for some $\fp$.
Non regular elements are excluded in his formulation of the Boulder lemma. 
Moreover Langlands assumes $\oD$ is obtuse.}.

\begin{proposition}\label{plex} 
Let ${\mathcal P(\oD)}$ be the set of ordered partitions of $\oD$. Then on has the following identity
$$\sum_{\fp\in\mathcal P(\oD)}(-1)^{\alpha^\L_\fp}\phi_\fp^\L(H)=
d_\oD(\L)+\sum_{\fp\in\mathcal P(\oD)}(-1)^{\beta^\L_\fp}\psi_\fp^\L(H)\ptf$$
\end{proposition}

\begin{proof}
Langlands says that the proof is a ``pleasant exercise''.
We observe that when $\L=0$ 
 then  $\phi_\fp^\L=\psi_\fp^\L$ and $\alpha^\L_\fp=\beta^\L_\fp+2a_\fp^1$ for all $\fp$;
now since $p\ge1$ we have
$d_\oD(\L)=0$ and the proposition is trivially true. For arbitrary $\L$
a proof is given at the end of section~\ref{prob} as
an immediate consequence of Corollaries~\ref{combi} and \ref{part}. 
\end{proof}

\section{Matrix equations}\label{matrix}

\subsection*{A slightly more general setting}
Consider a euclidean vector space 
$V_0$, a basis $\oD_0$
and two subsets $P\subset Q\subset\oD_0$.
Let $\oD_P^Q$ be the projection of $Q$ on the orthogonal of $P$ and $V_P^Q$ the 
subspace generated by $\oD_P^Q$. 
We denote by $\D_P^Q$ the bijection between $\oD_P^Q$ and its dual basis $\hD_P^Q$.
For $H\in V$ we denote by $H_P^Q$ the orthogonal projection of $H$ on $V_P^Q$.
With no loss of generality, we may assume that $\oD=\oD_P^Q$
and $V=V_P^Q$ for a pair of subsets $P\subset Q\subset\oD_0$
for some $\oD_0$. This allows us to use Arthur's notation introduced in 
\cite{A1}, \cite{A2} and \cite{A3} where $\oD_0$ is the set of simple roots
attached to a minimal parabolic subgroup $P_0$ in a reductive group $G$.
Using that standard parabolic subgroups
are bijectively attached to subsets (that may be empty) of $\oD_0$
and that this bijection is compatible with inclusion then,
by abuse of notation, using the same letter for a subset $P$ of $\oD_0$ and the
standard parabolic subgroup it defines, we recover Arthur's setting and notation.
Hence one may forget about $G$, and consider $P_0$ as the empty subset in $\oD_0$. 
We have to give new proof of some of the classical results since we shall
not assume $\oD_0$ obtuse.

\subsection*{Functions $\oT_{P,Q}^\L$ and $\hT_{P,Q}^\L$}
For $P\subset Q$ 
let $\oT_{P,Q}^\L$ be the characteristic function of the set of $H\in V_0$ such that
{for} $\omu\in\oD_{P}^{Q}$
{and} $ \hmu=\D_P^Q(\omu)$
$$\omu(H)\le0\coma{if }\hmu(\L)>0
\coma{and}\omu(H)>0\coma{otherwise}$$
and let $\hT_{P,Q}^\L$ be~\footnote{Our  $\hT_{P,Q}^\L(H)$ is denoted
$\phi_P^Q(\L,H)$ by Arthur in \cite[page 940]{A2}.
Unfortunately Arthur's notation is in conflict with Langlands one in \cite{L2} and
since we stick to Langlands' notation 
we had to introduce a different one.}
the characteristic function of the set of $H$ such that
$$\hmu(H)\le0\coma{if }\omu(\L)>0
\coma{and}\hmu(H)>0\coma{otherwise.}
$$
By convention $\oT_{P,Q}^\L=\hT_{P,Q}^\L=0$ when $P\not\subset Q$ and $\oT_{P,P}^\L=\hT_{P,P}^\L=1$.

 When $-\L$ is in the closure of the positive chamber with respect to an obtuse basis $\oD_P^Q$
 (which implies $\hmu(\L)\le0$ for all $\hmu\in\hD_P^Q$)
and in all cases when $\L=0$, we have
$$\oT_{P,Q}^\L=\tau_P^Q\com{and}\hT_{P,Q}^\L=\widehat\tau_P^Q$$
where $\tau_P^Q$ and $\hat\tau_P^Q$ are Arthur's functions. 
More generally, one can express the $\oT_{P,Q}^\L$
and $\hT_{P,Q}^\L$ in terms of the $\tau_P^Q$ and $\widehat\tau_P^Q$. 
To do this we need some notation. 
Let $P_\L$ (or  $P^Q_\L$ if some confusion may arise)
be the subset of $\oD_0$ such that $\hD_{P_\L}^{Q}$ is the set of
$\hmu\in\hD_P^Q$ with $\hmu(\L)>0$.
Similarly, let $Q^\L$ (or  $Q_P^\L$ if some confusion may arise) be the subset of $\oD_0$  
such that $\oD_P^{Q^\L}$ is the set of $\omu\in\oD_P^Q$ with $\omu(\L)>0$.
\begin{lemma}\label{idit}  
$$\oT_{P,Q}^\L=\sum_{P_\L\subset S\subset Q}(-1)^{a_{P_\L}^S}\,\tau_P^S\comm{and}
\hT_{P,Q}^\L=\sum_{P\subset S\subset Q^\L}(-1)^{a_S^{Q^\L}}\,\widehat\tau_S^Q$$
\end{lemma}
\begin{proof}
This is an immediate consequence of the binomial identity.
\end{proof}

\begin{lemma}\label{basic} 
$$\sum_{P\subset Q\subset R}(-1)^{a_P^Q}\tau_P^Q\hat\tau_Q^R=\delta_{P,R}$$
where $\delta_{P,R}$ is the Dirac symbol (it is zero if $P\ne R$ and $1$ if $P=R$). 
\end{lemma}
When $\oD_P^R$ is an obtuse basis this is \cite[Proposition 1.7.2]{LW} which itself is a special case of
\cite[Lemma 6.1]{A2}. 
We now give a proof which does not use this assumption~\footnote{Casselman has already 
observed, some years ago, that the obtuse assumption is not necessary (private communication).}.

\begin{proof}  
Fix $H\in V_0$. As above, but with $H$ in place of $\L$, we introduce  subsets
$P_H\subset \oD_0$ and $R^H\subset\oD_0$.
Then
$$\tau_P^Q(H)\hat\tau_Q^R(H)=1$$ if and only if ${P_H}\subset Q\subset {{R^H}}$ and hence
$$\sum_{P\subset Q\subset R}(-1)^{a_P^Q}\tau_P^Q(H)\hat\tau_Q^R(H)
=\sum_{{P_H}\subset Q\subset {{R^H}}}(-1)^{a_P^Q}\ptf$$
The binomial identity tells us this alternating sum vanishes unless ${P_H}={{R^H}}$. 
Assume it does not vanish, then $H_P^R\ne0$ if $P\not= R$. Now consider $\omu\in\oD_P^R$ 
and $\hmu\in\hD_P^R$ such that
$\hmu=\D_P^R(\omu)$ then, since ${P_H}={{R^H}}=Q$, we have
$$\omu(H)>0\comm{and}\hmu(H)\le0\quad\hbox{for}\quad \omu\in\oD_{P}^{Q}
$$
while
$$\omu(H)\le0\comm{and}\hmu(H)>0\quad\hbox{for}\quad \hmu\in\hD_Q^R$$
and hence
$$\lb H_P^R,H_P^R\rb=\sum_{\{\omu\in\oD_{P}^{R}\,|\,
\hmu=\D_P^R(\omu)\}}\omu(H)\hmu(H)\le0$$
which contradicts the non vanishing of $H_P^R$ unless $P=R$.
\end{proof}

\subsection*{Matrices $\bs\oT_{}^\L$ and $\bs\hT{}^\L$}
We shall now consider matrices  whose entries 
are indexed by pairs $(P,Q)$ of subsets of $\oD_0$.
Given matrices $\bs A$ with entries $\bs A_{P,Q}$
and $\bs B$ with entries $\bs B_{P,Q}$ the product $\bs A\bs B$ is the matrix 
$\bs C$ with entries
$$\bs C_{P,Q}=\sum_{R\subset\oD_0}\bs A_{P,R}\,\bs B_{R,Q}\ptf$$
Let $a_P^Q$ be the cardinal of $\oD_P^Q$ (so that $a_{P_0}^P$ is the cardinal of $P$),
$\ob_{P,Q}^\L$
 the cardinal of the set of $\hmu\in\hD_P^Q$ such that $\hmu(\L)\le0$
 and $\hb_{P,Q}^\L$ the cardinal of the set of $\omu\in\oD_P^Q$ such that $\omu(\L)\le0$. Let 
$$\og_{P,Q}^\L=a_{P_0}^P+\ob_{P,Q}^\L\com{and}
\hg_{P,Q}^\L=a_{P_0}^P+\hb_{P,Q}^\L\ptf$$
Consider matrices $\bs\oT_{}^\L$ and $\bs\hT{}^\L$
with entries 
$$\bs\oT_{P,Q}^\L=(-1)^{\og_{P,Q}^\L}\oT_{P,Q}^\L\qquad \com{resp.} 
\bs\hT{}^\L_{P,Q}=(-1)^{\hg_{P,Q}^\L}\hT_{P,Q}^\L\ptf$$ 
We shall be interested in the products $\bs A^{\Lu,\Ld}=\bs\oT{}^\Lu\bs\hT{}^\Ld$ and 
$\bs B^{\Lu,\Ld}=\bs\hT{}^\Ld \bs\oT^\Lu$.

\begin{lemma}\label{basic2}
Assume either $\L_i=0$ or 
$\oD_0$ is an obtuse basis and the $-\L_i$ 
are in the closure of the positive chamber.
Then 
$$\bs A^{\Lu,\Ld}=\bs\oT{}^\Lu\bs\hT{}^\Ld=\bs\hT{}^\Ld \bs\oT^\Lu=\bs B^{\Lu,\Ld}=\bs1\ptf
\leqno(1)$$
In particular
$$\sum_{P\subset Q\subset R}(-1)^{a_P^Q}\hat\tau_P^Q\tau_Q^R=\delta_{P,R}\ptf
\leqno(2)$$
\end{lemma}
\begin{proof} The assumptions on $\L=\L_i$ imply 
$\oT_{P,Q}^\L=\tau_P^Q\com{and}\hT_{P,Q}^\L=\widehat\tau_P^Q$.
Then, Lemma \ref{basic} yields
$$\bs A_{P,R}^{\Lu,\Ld}=
\sum_{P\subset Q\subset R}(-1)^{\og_{P,Q}^\Lu-\hg_{Q,R}^\Ld}\,\,\oT_{P,Q}^\Lu\hT_{Q,R}^\Ld
=\sum_{P\subset Q\subset R}(-1)^{a_P^Q}\tau_P^Q\hat\tau_Q^R=\delta_{P,R}$$
and hence $\bs\oT{}^\Lu\bs\hT{}^\Ld=\bs1$ which proves $(1)$.
Assertion $(2)$ is an explicit form of equation
$\bs\hT{}^\Ld \bs\oT^\Lu=\bs1$.
 \end{proof}
\subsection*{The key ingredient}
The next proposition and its proof is a generalization of
Arthur's Lemma 6.3 in  \cite{A2} (which appears below as Corollary~\ref{part}).

\begin{proposition}\label{combo} 
$$\bs B^{\Lu,\Ld}_{P,R}=(-1)^{a_P^{P^\Lu}}\delta_{P_\Lu,R^\Ld}\ptf$$
\end{proposition}
\begin{proof} 
By definition
 $$\bs B_{P,R}^{\Lu,\Ld}=
 \sum_{P\subset Q\subset R}(-1)^{\hg_{P,Q}^\Ld-\og_{Q,R}^\Lu}\,\,\hT_{P,Q}^\Ld\oT_{Q,R}^\Lu \ptf$$
 In view of Lemma \ref{idit} we have
$$\bs B_{P,R}^{\Lu,\Ld}=\!\!\sum_{P\subset Q\subset R}(-1)^{\hg_{P,Q}^\Ld-\og_{Q,R}^\Lu}
\sum_{P\subset S\subset Q_P^\Ld}(-1)^{a_S^{Q_P^\Ld}}\widehat\tau_S^Q 
\sum_{Q_\Lu\subset {S'}\subset R}(-1)^{a_{Q^R_\Lu}^{{S'}}}\tau_Q^{S'}$$
Since  $a_{Q}^{Q^R_\L}=\ob_{Q,R}^\L$ and  $a_{Q_P^\L}^{Q}=\hb_{P,Q}^\L$ we see that
$$\bs B_{P,R}^{\Lu,\Ld}=\!\!\sum(-1)^{a_P^Q}
(-1)^{a_S^{{S'}}}\widehat\tau_S^Q \tau_Q^{S'}$$
where the sum runs over triples $(S,Q,S')$ verifying the inclusions
$${P\subset S\subset Q_P^\Ld\subset Q\subset Q^R_\Lu\subset {S'}\subset R}\ptf\leqno(\star)$$
Observe that $Q^R_\Lu\supset P^R_\Lu$ and $Q_P^\Ld\subset R_P^\Ld$. This shows that
 given any  pair $(S, S')$ with
$S\subset R_P^\Ld$ and $S'\supset P^R_\Lu$
we have the inclusions $(\star)$ for any $Q$ with $S\subset Q\subset {S'}$. 
Then
$$\bs B_{P,R}^{\Lu,\Ld}=\!\!\sum_{\{S,S'\,|\,S\subset R_P^\Ld\,,\,\,S'\supset P^R_\Lu\}}(-1)^{a_P^{S'}}
\sum_{\{Q\,|\,S\subset Q\subset S'\}}(-1)^{a_S^Q}\,\widehat\tau_S^{\,Q}\, \tau_Q^{S'}\ptf$$
It follows from equation $(2)$ in Lemma \ref{basic2} 
that the sum over $Q$ vanishes unless $S={S'}$ and hence we have
$$\bs B_{P,R}^{\Lu,\Ld}=\!\!\sum_{\{S\,|\,P^R_\Lu\subset S\subset R_P^\Ld\}}(-1)^{a_P^S}\ptf$$
This in turn, by the binomial identity,  vanishes unless $P_\Lu=R^\Ld$. 
 \end{proof}

 \begin{corollary} \label{combi}
 The matrices
$\bs\oT^\L$ and $\bs\hT{}^\L$ are inverse of each other.
\end{corollary} 

\begin{proof} Proposition~\ref{combo} shows that $\bs B^{\L,\L}_{P,R}$ 
vanishes unless $P_\L=R^\L$. But,
as in the proof of Lemma \ref{basic} we see this condition implies $\lb \L_P^R,\L_P^R\rb\le0$. 
Hence, if $\bs B^{\L,\L}_{P,R}$ does not vanish
we have $\L_P^R=0$ and then $P=R^\L$ and $P_\L=R$ which yields $P=R$.
\end{proof}

\begin{corollary} \label{part} 
$$\sum_{P\subset Q\subset R}(-1)^{\hb_{P,Q}^\L}\,\,\hT_{P,Q}^\L\,\,\tau_Q^R=d_{\oD_P^R}(\L)\ptf$$
\end{corollary}

\begin{proof}  
Consider the particular case $\Lu=0$ and  $\Ld=\L$. Then $P_\Lu=R$ and Proposition \ref{combo} shows that
$B_{P,R}^{\,0,\L}=0$ unless $R=R^\L$ which holds if and only if $\L$ 
is in the positive chamber with respect to $\oD_P^R$ so that
$$B_{P,R}^{\,0,\L}=(-1)^{a_P^R}d_{\oD_P^R}(\L)\ptf$$
It remains to observe that
$$\hg_{P,Q}^\Ld-\og_{Q,R}^\Lu=a_{P_0}^P+\hb_{P,Q}^\L-a_{P_0}^Q-\hb_{Q,R}^0\equiv a_P^R+\hb_{P,Q}^\L\com{(modulo 2)}
\ptf$$
\end{proof}

\section{Application to the Boulder lemma}\label{prob}

\subsection*{Where ordered partitions disappear} 
Consider an ordered partition $\fp$ of $\oD_P^R$
such that $F_\fp^1=\oD_P^R\bsl\oD_P^Q$ with $P\subsetneq Q\subsetneq R$.
Let $\fq$ be the ordered partition of $\oD_P^Q$ defined by
$$F_\fq^i=F_\fp^{i+1}
\com{for}1\le i\le r(\fq)=r(\fp)-1\ptf$$ 
Then, Langlands' characteristic functions $\phi_\fp^\L$ and $\psi_\fp^\L$ may be rewritten:
$$\phi_\fp^\L=\oT_{Q,R}^\L\,\phi_\fq^\L
\com{and}
\psi_\fp^\L=\tau_{Q}^{R}\,\phi_\fq^\L\ptf\leqno(\star)$$
Moreover
$$
b_\fp^\L=\ob_{Q,R}^\L+b_\fq^\L
\com{and}\alpha^\L_\fp=\alpha^\L_\fq+a_P^Q+1+\ob_{Q,R}^\L\ptf\leqno(\star\star)$$
The next proposition shows how to replace, up to a sign, alternating sums 
over ordered partitions of functions
$\phi_\fp^\L$ by a single characteristic function.  
\begin{proposition} \label{simple}
Let $\hb^\L_{P,R}$ be the number of $\omu\in\oD_P^R$ such that $\omu(\L)\le0$.
Then,
$$\sum_{\fp\in\mathcal P(\oD_P^R)}(-1)^{\alpha^\L_\fp}\phi_\fp^\L=(-1)^{\hb^\L_{P,R}}\hT_{P,R}^\L\ptf$$
\end{proposition}

\begin{proof} The above remarks $(\star)$ and $(\star\star)$ show that
$$\sum_{\fp\in\mathcal P(\oD_P^R)}(-1)^{\alpha^\L_\fp}\phi_\fp^\L=
\sum_{P\subset Q\subsetneq R}(-1)^{a_P^Q+1+\ob_{Q,R}^\L}\,\,\oT_{Q,R}^\L
\sum_{\fq\in\mathcal P(\oD_P^Q)}(-1)^{\alpha^\L_\fq}\phi_\fq^\L\ptf
$$
The equivalence of Proposition \ref{simple}
and Corollary \ref{combi} now follows by induction on the cardinal of 
$\oD_P^R$. 
\end{proof}
\subsection*{Proof of the Boulder lemma}
Proposition~\ref{plex} claims an identity that in view of $(\star)$  can be written
$$\sum_{\fp\in\mathcal P(\oD_P^R)}(-1)^{\alpha^\L_\fp}\phi_\fp^\L(H)=d_{\oD_P^R}(\L)-\sum_{P\subset Q\subsetneq R}\tau_Q^R(H)
\sum_{\fq\in\mathcal P(\oD_P^Q)}(-1)^{\alpha^\L_\fq}\phi_\fq^\L(H)\ptf$$
To prove it we have to show that, thanks to Proposition~\ref{simple},
$$(-1)^{\hb^\L_{P,R}}\hT_{P,R}^{\L}(H)=d_{\oD_P^R}(\L)\,-\!\!\!
\sum_{P\subset Q\subsetneq R}\!\!(-1)^{\hb_{P,Q}^\L}\tau_Q^R(H)\hT_{P,Q}^\L(H)$$
but this is nothing than Corollary~\ref{part}.

\section{Truncated Eisenstein series}\label{trunc}
\subsection*{ Series $E''$ versus truncated Eisenstein series}
Here we borrow the notation from \cite{LW}.
Let $G$ be a reductive group over a number field $F$ with minimal parabolic subgroup $P_0$.
Let $P$ be a standard parabolic subgroup with Levi $M$,
 $S$ an associate standard parabolic subgroup and $\Phi$ a cuspidal form on $\mathbf{X}_P$
 and $T$ some parameter in $\ga_0$. Let $\oD=\oD_P^G$
 and consider $\L\in V\otimes\CM$ such that $\Re(\L)$ is far enough inside the positive Weyl chamber (with respect to $\oD$). 
 Langlands introduces in \cite[\S9]{L2} functions 
$$F^{''}_S(x,\Phi,\L)=\sum_{s\in W(\ga_P,\ga_S)}\sum_{\fp\in\mathcal P(\oD)}(-1)^{\alpha_\fp^{\Re(s\L)}}
\phi_\fp^{\Re(s\L)}(H_S(x)-T_S)(M(s,\L)\Phi)(x,s\L)$$
where, by definition,
$$\Phi(x,\L)=e^{<\L+\rho_P,H_P(x)>}\Phi(x)\ptf$$
The reader is warned that 
the dependence on $T$ (denoted $H_0$ by Langlands in \cite{L2}) does not appear 
explicitly on the left hand side.
Then, Langlands introduces
$$E^{''}(x,\Phi,\L)=\sum_S\sum_{\xi\in S(F)\bsl G(F)}F_S^{''}(\xi x,\Phi,\L)$$
where the sum over $S$ runs over parabolic subgroups associated to $P$.
On the other hand Arthur has defined in \cite{A3} a truncation operator denoted $\trunc$.
When applied to the Eisenstein series 
$$E(x,\Phi,\L)=\sum_{\gamma\in P(F)\bsl G(F)}\Phi(\gamma x,\L)$$
then, according to \cite[Proposition 5.4.1]{LW},
one obtains the following expression
$$\trunc\! E(x,\Phi,\L)\!=\!\sum_S\!\!\sum_{s\in W(\ga_P,\ga_S)}\sum_{\xi\in S(F)\bsl G(F)}\!\!\!\!\!\!\!\!\!
(-1)^{a(s)}\!\phi_{M,s}(s\mun\!(H_S(\xi x)-T_S))(M(s,\L)\Phi)(\xi x,s\L)$$
where, by definition,
$$\phi_{M,s}(H)=\hT_{P,G}^{\,\,\Re(s\L)}(sH)$$
and the integer $a(s)$ is the number of simple roots $\omu\in\oD_S$ such that $\omu(s\L)<0$.

\begin{proposition}\label{pareil}
$$E^{''}(x,\Phi,\L)=\trunc\! E(x,\Phi,\L)\ptf$$ 
\end{proposition}

\begin{proof} It suffices to observe that if $\Re(\L)$ is in the positive chamber,
then for $H\in \ga_P$ and $s\in W(\ga_P,\ga_S)$ one has 
$$\sum_{\fp\in\mathcal P(\oD)}(-1)^{\alpha_\fp^{\Re(s\L)}}\phi_\fp^{\Re(s\L)}(sH)=(-1)^{a(s)}\phi_{M,s}(H)\ptf$$
This assertion is a particular case of Proposition~\ref{simple} 
where $\oD_0$ is the set of simple roots of $G$
with respect to the minimal parabolic subgroup $P_0$ and $\oD=\oD_P^G$.
 \end{proof}

\subsection*{On scalar products}
A formula for the scalar product of two series $E''$
is stated by Langlands \cite[Section 9]{L2} without proof.
Arthur proves a formula \cite[Lemma 4.2]{A3} for the scalar product of two truncated Eisenstein series
which is said (again without proof) to be equivalent to Langlands' one. 
As already observed this equivalence is not obvious since ordered partitions 
that show up in Langlands' formula never occur in Arthur's statement.
Now, this follows from \ref{pareil}. This equivalence of the formulas for scalar products can also be
proved directly using again Proposition~\ref{simple}. We leave it to the reader.
We observe that a proof of the scalar product formula simpler than in \cite{A3}
is given in \cite[Chapter 5]{LW}.

\bibliographystyle{amsalpha}
\bibliography{pleasant}

\def\bysame{\leavevmode ---------\thinspace}
\makeatletter\if@francais\providecommand{\og}{<<~}\providecommand{\fg}{~>>}
\else\gdef\og{``}\gdef\fg{''}\fi\makeatother
\def\cdrandname{\&}
\providecommand\cdrnumero{no.~}
\providecommand{\cdredsname}{eds.}
\providecommand{\cdredname}{ed.}
\providecommand{\cdrchapname}{chap.}
\providecommand{\cdrmastersthesisname}{Memoir}
\providecommand{\cdrphdthesisname}{PhD Thesis}
\begin{thebibliography}{1}

\bibitem{A1}
{\scshape J.~Arthur}, {\og The characters of discrete series as orbital
  integrals\fg}, \emph{Invent. Math.} \textbf{32} (1976), \cdrnumero 3,
  p.~205-261.

\bibitem{A2}
\bysame , {\og A trace formula for reductive groups. I. Terms associated to
  classes in $G({\QM})$\fg}, \emph{Duke Math.} \textbf{45} (1978), \cdrnumero
  4, p.~911-952.

\bibitem{A3}
\bysame , {\og A trace formula for reductive groups. II. Applications of a
  truncation operator\fg}, \emph{Compositio Math.} \textbf{40} (1980),
  \cdrnumero 1, p.~87-121.

\bibitem{BW}
{\scshape A.~Borel {\normalfont \cdrandname}~N.~Wallach}, \emph{Continuous
  Cohomology, Discrete Subgroups and Representations of Reductive groups},
  Annals of Math. Studies, vol.~94, Princeton University Press, 1980.

\bibitem{LW}
{\scshape J.-P. Labesse {\normalfont \cdrandname}~J.-L. Waldspurger}, \emph{La
  formule des traces tordue d'aprs le Friday Morning Seminar. With a foreword
  by Robert Langlands}, CRM Monograph Series, vol.~31, American Mathematical
  Society, Providence, 2013.

\bibitem{L2}
{\scshape R.~P. Langlands}, {\og Eisenstein series\fg}, \emph{Proc. Sympos.
  Pure Math.} \textbf{IX} (1966), p.~235-252.

\bibitem{L1}
\bysame , \emph{On the Functional Equations Satisfied by Eisenstein Series},
  Lecture Notes in Mathematics, vol. 544, Springer-Verlag, Berlin-New York,
  1976.

\bibitem{L3}
\bysame , {\og On the Classification of Irreducible Representations of Real
  Algebraic Groups\fg}, \emph{Math. Surveys and Monographs} \textbf{31} (1989),
  p.~101-170.

\end{thebibliography}

\end{document}